\documentclass[final,1p, times]{elsarticle}
\usepackage{amsmath, amssymb, amsthm}
\usepackage{url}

\newtheorem{theorem}{Theorem}[section]
\newtheorem{lemma}[theorem]{Lemma}
\newtheorem{corollary}[theorem]{Corollary}

\newtheorem{proposition}[theorem]{Proposition}
\newtheorem{remark}[theorem]{Remark}

\newtheorem{question}[theorem]{Question}

\newcommand{\N}{\ensuremath{\mathbb{N}}}
\newcommand{\sH}{\ensuremath{\mathcal{H}}} 
\newcommand{\K}{\kappa} 
\newcommand{\SO}{\mathcal{O}} 

\newcommand{\dpl}{\ensuremath{\mathcal{D}}}
\newcommand{\smpl}{\ensuremath{\mathcal{S}}}

\renewcommand{\l}{\ell}

\def\cX{\mathcal X}

\begin{document}

\begin{frontmatter}

\title{On the structure of numerical sparse semigroups and applications to Weierstrass points}

\author[ac]{Andr\'{e} Contiero\fnref{fn1}}
\ead{andrecontiero@mat.ufal.br}

\author[gugu]{Carlos Gustavo T. A. Moreira}
\ead{gugu@impa.br}

\author[pmv]{Paula M. Veloso\fnref{fn2}}
\ead{pmveloso@id.uff.br}

\fntext[fn1]{Partially supported by PPP-FAPEAL (20110901-011-0025-0048), by Programa de P\'{o}s-doutorado de Ver\~{a}o 2012, IMPA, and by Projeto Universal CNPq (486468/2013-5)}

\fntext[fn2]{Partially supported by Programa Primeiros Projetos (CEX - APQ-03972-10)-- FAPEMIG, by Programa de P\'{o}s-doutorado de Ver\~{a}o 2012, IMPA, and by Programa ``Para Mulheres na Ci\^{e}ncia'' -- L'Or\'{e}al, UNESCO, Academia Brasileira de Ci\^{e}ncias}

\address[ac]{Instituto de Matem\'atica, Universidade Federal de Alagoas (UFAL).
Avenida Lourival de Melo Mota,
57072-970, Macei\'o -- AL, Brazil}

\address[gugu]{Instituto de Matem\'atica Pura e Aplicada (IMPA). 
Estrada Dona Castorina 110, 22460-320, Rio de Janeiro -- RJ, Brazil}

\address[pmv]{Instituto de Matem\'{a}tica e Estat\'{i}tica, Universidade Fe\-de\-ral Fluminense (UFF). Rua M\'ario Santos Braga S/N, Campus do Valonguinho, 24020-140
Niter\'{o}i -- RJ, Brazil}

\begin{abstract}
In this work, we are concerned with the structure of sparse semigroups and some applications of them to Weierstrass points. 
We manage to describe, classify and find an upper bound for the genus of sparse semigroups. We also study the realization 
of some sparse semigroups as Weierstrass semigroups. The smoothness property of monomial curves associated to 
(hyper)ordinary semigroups presented by Pinkham and Rim-Vitulli, and  the results on double 
covering of curves by Torres are crucial in this.
\end{abstract}


\end{frontmatter}

\section{Introduction}

Let $\sH$ be a numerical semigroup of genus $g>1$. We say that
$\sH$ is a \emph{sparse semigroup} if every two subsequent gaps of $\sH$
 are spaced by at most $2$. The concept of a sparse semigroup was introduced
by Munuera--Torres--Villanueva in \cite{MunueraTorresVillanueva} and emerged as a 
generalisation of \emph{Arf semigroups}. The latter appear naturally in the study of
 one-dimensional analytically unramified domains by analysing their
 valuation semigroups (see \cite{Arf},n\cite{BarucciDobbsFontana} and \cite{Lipman} 
 for further details on Arf semigroups). 
Furthermore, one of the main subjects related to numerical semigroups are 
Weierstrass points on algebraic curves (points whose gap sequence of the numerical semigroup
associated to a smooth projective pointed curve is the sequence of orders of vanishing 
of the holomorphic differentials of the curve at the base point.)

In this work, we are concerned with studying the structure of sparse semigroups and some
applications to Weierstrass points. Looking for a classification and for an upper bound for genus of sparse 
semigroups, we introduce a \emph{leap} count assertion (Theorem \ref{leapthm}), which involves an 
interplay between single and double leaps. Besides, it plays a fundamental role in the main results of this work. 

Section \ref{SparseSec} presents several consequences of Theorem \ref{leapthm}. We prove that, if the 
genus of a sparse semigroup is large enough, then the last few gaps are spaced by $2$. This
results is proved regardless the parity of the Frobenius number $\l_{g}$ (Proposition \ref{finalDblLeapsProp}). 
Additionally, we classify some sparse semigroups with few single leaps or with large Frobenius number. 
At this point, \emph{(hyper)ordinary semigroups} (see \cite{Rim-Vitulli}) and $\gamma$-\emph{hyperelliptic semigroups} 
show up (see \cite{Torres}).

Looking for an upper bound for the genus of a sparse semigroup we introduce, in Section \ref{limitsparse}, 
the concept of a \emph{limit sparse semigroup}: sparse semigroups with as many single as double leaps. 
Considering the parity of the Frobenius number, we classify limit sparse semigroups with even Frobenius 
number (Theorem \ref{charactSparseLimTheo}), which are all hyperordinary with multiplicity $3$. 
As a consequence, we get an upper bound for the genus of any sparse semigroup with even Frobenius 
number, namely $g<4r$ where $\l_{g}=2g-2r$ (Corollary \ref{genusQuotaSparseTheo}). 

We also classify limit sparse semigroups with odd Frobenius number. In this case, the \emph{multiplicity} of 
the semigroup plays an import role. If the multiplicity of the limit sparse semigroup $\sH$ is even, then $\sH$ 
is an $r$-hyperelliptic semigroup, where $\l_{g}=2g-2r-1$, (Theorem \ref{oddFrobEvenMultTheo}). On the 
other hand, if the multiplicity is odd, then either $\sH$ is hyperordinary of multiplicity $3$, or a semigroup of 
multiplicity $2r+1$, where $\l_{g}=2g-2r-1$ is the Frobenius number of $\sH$ (Theorem \ref{oddFrobOddMul}). 
With the classification of limit sparse semigroups with odd Frobenius number in mind, we find an upper bound 
for the genus of these sparse semigroups, namely $g\leq4r+1$, except when all nongaps smaller than the 
Frobenius number are even (Corollary \ref{genusQuotaSparseOddTheo}).

Finally, in the last section, we study the realization of (limit) sparse semigroups as Weierstrass semigroups. 
At this point the smoothness property of monomial curves associated to (hyper)ordinary semigroups 
presented by Pinkham \cite{Pinkham} and Rim-Vitulli \cite{Rim-Vitulli} is crucial. Furthermore, 
regarding $\gamma$-hyperelliptic sparse semigroups, the results by Torres \cite{Torres} 
on double covering of curves are applied.

\section{Sparse semigroups}\label{SparseSec}

Let $\N$ be the set of natural numbers. A \emph{numerical semigroup} 
$\sH = \{0 = n_{0} < n_{1} < \ldots \} \subseteq \N$ of finite genus $g\geq 1$ is an additive 
subset of $\N$ containing $0$, closed under addition and such that there are only $g$
elements in the set $\N\setminus\sH=\{1=\l_1<\l_2<\dots<\l_{g}\}$. The elements in 
$\N \setminus \sH$ are called \emph{gaps} and the largest gap $\l_{g}$ is called the
\emph{Frobenius number} of $\sH$.  The elements of $\sH$ are referred to as \emph{nongaps}, 
 and the smallest positive nongap is said to be the \emph{multiplicity} of $\sH$. 

A \emph{sparse} (numerical) \emph{semigroup} $\sH$ is a numerical semigroup where two subsequent
gaps of $\sH$ with $1\leq\l_{i-1},\l_{i}\in\{\l_1 < \ldots < \l_g\}$ are spaced by at most $2$,
\begin{equation*}
\l_{i}-\l_{i-1}\leq 2, \, i=2,\ldots,g \ \ , \ \l_i\in\N\setminus\sH .
\end{equation*}
Equivalently, $\sH$ is sparse if its first $c-g$ nongaps satisfy $$n_{i + 1} - 
n_ i \geq 2, \, i = 1,\ldots , c - g\ ,$$ where $c:=\l_g+1$ is the least integer 
such that $c+h\in \sH$ for every $h\in \N$. The integer $c$ is said to be the 
\emph{conductor} of $\sH$ (clearly, $c=n_{c-g}$).

Two particular classes of sparse semigroups will appear frequently in this work: \emph{ordinary sparse semigroups} 
($\sH_g=\{0,g+1,g+2,\dots\}$) and \emph{hyperordinary sparse semigroups} ($\sH=m\N+\sH_{g}$, $0<m<g$).

Another class of sparse semigroups are Arf semigroups \cite[Corollary 2.2]{MunueraTorresVillanueva}. We recall that
a numerical semigroup $\sH$ is an \emph{Arf semigroup} if
$n_{i} + n_{j} - n_{k} \in \sH$, for $i \geq j \geq k$ (see \cite[Theorem I.3.4]{BarucciDobbsFontana} for fifteen alternative
characterizations of Arf semigroups, among which we call attention to the following: $\sH$
 is an Arf semigroup if and only if $2n_i-n_j\in \sH$, for all $i\geq j\geq 1$). There are, however,  sparse semigroups that 
 are not Arf see Remark \ref{nonArf}  or \cite[Example 2.3]{MunueraTorresVillanueva}).

It is well known that for any numerical semigroup the Frobenius number $\l_{g}$ 
satisfies $\l_{g} \leq 2g - 1$ (see, for instance, \cite[Theorem 1.1]{Oliveira}).
We may, thus, define the parameter $$\K: = 2g - \l_{g}>1\, ,$$
and we notice that $\K\leq g$.

Since sparse semigroups are the ones where subsequent gaps are either consecutive
or spaced by $2$,  it is only natural to count how many pairs of subsequent gaps
are in either situation. Given a sparse semigroup $\sH$, consider the sets:
$$\dpl := \{ i \, ; \, \l_{i + 1} - \l_{i} = 2 \} \ \text{ (``double leaps'')},$$
$$\smpl := \{ i \, ; \, \l_{i + 1} - \l_{i} = 1 \} \ \text{ (``single leaps'')},$$
\noindent and their cardinalities:
$$D :=\# \dpl \ \ \text{and}  \ \ S :=\#  \smpl.$$

\begin{theorem}\label{leapthm}
Let $\sH$ be a sparse semigroup of genus $g$. Then:
\begin{enumerate}
\item $D + S = g - 1$.
\item $D = g - \K$.
\item $S = \K - 1$.
\end{enumerate}
\end{theorem}

\begin{proof}
(1): Every gap, except the last one, $\l_{g}$, is the starting point of a leap. 
So the total number of leaps, either single or double, is $g - 1$. Thus, 
$D + S = g - 1$.

(2) and (3): Between $1$ and $\l_{g} = 2g - \K$, there are $S$ single leaps and 
$D$ double leaps, regardless of their order. So, $\l_{g} = 2g - \K = 1 + S +2D$. 
This equation together with the previous one yield the desired results.
\end{proof}

We shall denote leaps by an ordered pair of subsequent gaps $(\l_{i}, \l_{i+1})$, where $\l_{i} <  \l_{i+1}$. Clearly, a leap
is single if $\l_{i} + 1 = \l_{i+1}$, and double if $\l_{i} + 2 = \l_{i+1}$. 

Next proposition gives us a little bit more information on the structure of sparse semigroups.
It tells us that, if $g \geq 2\K - 1$, then the last few gaps occur every two integers. 
\begin{proposition}\label{finalDblLeapsProp}
Let $\sH$ be a sparse semigroup of genus $g$. If $g\geq 2\K-1$, then $\l_{i+1} - \l_{i} = 2$, 
for every $i=2\K-2,\ldots,g-1$.
\end{proposition}

\begin{proof}
If $g=2\K-1$, then $\sH$ has $\K-1$ single leaps. Let us assume that 
$\l_g-\l_{g-1}=1$. Thus $\l_{g-1}-n_i,\l_{g}-n_i$ are consecutive gaps for 
$i=1,\ldots,\K-1$. Since $\l_{g-1},\l_g$ are consecutive gaps, the total number 
of single leaps for $\sH$ is bigger than $\K-1$, which is a contradiction.
Now, if $\sH$ is a sparse semigroup of genus $g=2\K+j$, with $j\in\N$. Then 
$\widetilde{\sH}=\sH\cup\{\l_g\}$ is a sparse semigroup of genus 
$\widetilde{g}=2\K+j-1$. Thus the gaps of $\sH$ satisfy $\l_{i+1}-\l_{i}=2$,
for $i=2\K-2,\ldots,g-2$. Hence we just have to analyze $\l_{g}-\l_{g-1}$, which 
is analogous to the case where $g=2\K-1$.
\end{proof}

For even values of $\K$,the previous result had been stated and proved by Munuera, Villanueva and Torres 
\cite[Theorem 2.9 (3)] {MunueraTorresVillanueva} by means of a completely diverse approach.

Upon researching sparse semigroups, it became clear to us that those having 
genus $g = 2\K -1$ and Frobenius number $\l_{g} = 2g - \K = 3\K -2$ are quite 
special. In fact, the lemma below suggests that they are ``limit'' in some sense; 
this notion will become clearer in the next section.

\begin{lemma}\label{reductionLemma}
Let $\sH$ be a sparse semigroup of genus $g = 2\K + j$, $j \geq 0$, with 
Frobenius number $\l_g=2g - \K$. Then there is a sparse semigroup 
$\widetilde{\sH}$ of genus $\widetilde{g} = 2\K -1$ and Frobenius number 
$\l_{\widetilde{g}} = 2\widetilde{g} - \K = 3\K - 2$ such $\sH$ is a 
subsemigroup of $\widetilde{\sH}$. 
\end{lemma}

\begin{proof}
Since $g \geq 2\K-1$, Proposition \ref{finalDblLeapsProp} tells us that the
last $j + 2$ gaps of $\sH$ are spaced by $2$. Consider the set
$\widetilde{\sH} = \sH \cup \{\l_{g-j-2}, \l_{g-j-1},\dots,\l_{g}\}$. Clearly,
$\widetilde{\sH}$ contains $0$ and is additively closed. From the fact that
$\sH$ is sparse, we see that so is $\widetilde{\sH}$. So $\widetilde{\sH}$ is
a sparse semigroup and, by construction, it has $\widetilde{g} = 2\K - 1$ gaps,
$\sH$ is a subsemigroup of $\widetilde{\sH}$, and its Frobenius number is
$\l_{\widetilde{g}} = 2\widetilde{g} - \K = 3\K - 2$.

\end{proof}

\medskip

We present now several consequences and applications of our leap-count result
(Theorem \ref{leapthm}), which illustrate the techniques used in the theory of
sparse semigroups. It will be clear, trough the next results, that there are
only few sparse semigroups with large Frobenius number (or, equivalently, with
few single leaps). We will make this statement more precise in the next section.

We recall that a numerical semigroup is said to be \emph{symmetric}
(resp. \emph{quasi-symmetric}) if $\l_g=2g-1$ (resp. $\l_g=2g-2)$.

\begin{corollary}\label{symCor}
If $\sH$ is a symmetric sparse semigroup, then $\sH$ is the hyperelliptic
semigroup $\sH=\langle 2,2g+1 \rangle $.
\end{corollary}
\begin{proof}
Since $\l_{g}=2g-1$, the sparse semigroup $\sH$ does not have single leaps i.e. $\K=0$.
Then $2\in\sH$ and all the odd numbers between $1$ and $\l_{g}$ are gaps.
\end{proof}

\begin{corollary}\label{quasiSymCor}
If $\sH$ is a quasi-symmetric sparse semigroup, then, either $\sH= \langle 3,4,5 \rangle $, or $\sH= \langle 3,5,7 \rangle $.
\end{corollary}
\begin{proof}
Since $\l_g=2g-2$, we have that $\K = 2$ and so $S=1$ and $D=g-2$. We must
have $1,2 \not\in \sH$, which already accounts for the only single leap, so
$3 \in \sH$. Since all subsequent leaps must be double, the remaining gaps
must all be even numbers. Then $g\leq 3$. 
Now, notice that there are no numerical semigroups of genus $g = 1$, otherwise
$\l_g=\l_1=2g-2=0$, a contradiction. If $g=2$, we have $\l_g= \l_2=2g-2=2$, and
so $\sH= \langle 3,4,5 \rangle $. Finally, for $g=3$, we have $\l_g=2g-2=4$,
so $1$ and $2$ are also gaps, for they divide $4$, and thus $\sH= \langle 3,5,7 \rangle $.
\end{proof}

We say that a numerical semigroup is \emph{$\gamma$-hyperelliptic} if it has exactly 
$\gamma$ even gaps.  For the sake of clarity, we note that a $\gamma$-hyperelliptic semigroup
may have odd gaps and the integer $\gamma$ is not necessarily its genus.  Such semigroups
 are closely related with double covering of curves  \cite{BallicoCentina, Torres, Torres2}. 
 Additionaly, they arise when we deal with the
characterization of sparse semigroups having as many single as double leaps
(see next section).

\begin{theorem}\label{charactOdd3}
Let $\sH$ be a sparse semigroup having genus $g \geq 3$ and $\l_{g} = 2g - 3$. Then
 $\sH$ is one of the following:
\begin{enumerate}
\item $\sH =  3\N +\sH_{5}$, $\sH$ is $2$-hyperelliptic;
\item $\sH =  3\N +\sH_{7}$, $\sH$ is $2$-hyperelliptic;
\item $\sH = 2(\N \setminus \{1\}) \cup \sH_{2g - 2}$, $\sH$ is
 $1$-hyperelliptic. 
\end{enumerate}
\end{theorem}

\begin{proof}


Theorem \ref{leapthm} (3) tells us that $S=2$. Clearly, $\l_{1} = 1$ and $\l_{2} = 2$, which accounts for 
one single leap. If $3 \in \sH$, by the sparse property, we must have that
$1, 2, 4, 5 \not\in \sH$, and this accounts for all $2$ single leaps. 
Thus, either $\N \setminus \sH = \{1, 2, 4, 5\}$ and $g = 4$ (1), or $\N \setminus \sH = \{1, 2, 4, 5, 7\}$ 
and $g = 5$ (2). 
Otherwise, $\l_{3} = 3$, and all leaps from this point 
on must be double. So $\N \setminus \sH = \{1, 2, 3, 5, 7, \ldots 2g - 3\}$ and
$\sH = \{4, 6, \ldots , 2g - 4\} \cup \{n \in \N \, ; \, n \geq 2g - 2 \}$ (3).
\end{proof}

\begin{corollary}
Let $\sH$ be a numerical semigroup having genus $g \geq 6$ and $\l_{g} = 2g - 3$.
 Then the following are equivalent:
\begin{itemize}
\item[a.] $\sH$ is sparse;
\item[b.] $\sH$ is $1$-hyperelliptic.
\end{itemize}
\end{corollary}

\begin{proof}
The implication $a. \implies b.$ follows immediately from the previous theorem.
The other implication can be proved as follows: since $\sH$ is $1$-hyperelliptic, 
the only even gap of $\sH$ must be $2$. Indeed, if $\l>2$ were an even gap, then, from $\l=2+(\l-2)$, 
$2$ or $\l-2$  would be a smaller even gap, a contradiction. 
So, every even number 
larger than $2$  is a nongap, and thus $\l_r \ge 2r-3, \forall r \ge 3$, and, if 
the equality holds  for $r=g$, it must hold for every $r$ with $3 \le r \le 
g$, which implies that $\sH$ is sparse.
\end{proof}
 
\begin{theorem}\label{charactEven4}
Let $\sH$ be a sparse semigroup having genus $g \geq 4$ and $\l_{g} = 2g - 4$. Then
 $\sH$ is one of the following:
\begin{enumerate}
\item $\sH = 3\N +\sH_{8}$, and $\sH$ is $3$-hyperelliptic;
\item $\sH = 3\N +\sH_{10}$ , and $\sH$ is $4$-hyperelliptic;
\item $\sH = 4\N + \sH_{6}$, and $\sH$ is $2$-hyperelliptic;
\item $\sH = \sH_{4}$, and $\sH$ is $2$-hyperelliptic;
\item $\sH = 5\N +\sH_{6}$, and $\sH$ is $3$-hyperelliptic;
\item $\sH = \{0, 5, 7\} \cup \sH_8$, and $\sH$ is $4$-hyperelliptic;
\end{enumerate}
\end{theorem}

\begin{proof}
Theorem \ref{leapthm} (3) tells us that $S = 3$. Again, $\l_{1} = 1$ and $\l_{2} = 2$, 
which account for one single leap.

If $3 \in \sH$, since $\sH$ is sparse and $S=3$, we must have that
$1, 2, 4, 5, 7, 8 \not\in \sH$ accounting for all $3$ single leaps. Thus, either
$\N \setminus \sH = \{1, 2, 4, 5, 7, 8\}$ and $g = 6$ (1), or
$\N \setminus \sH = \{1, 2, 4, 5, 7, 8, 10\}$ and $g = 7$ (2). 

Otherwise, $\l_{3} = 3$, which accounts for a second single leap. 
If $4 \in \sH$, since $\sH$ is sparse, $5 \not\in \sH$. Notice that $6 \not\in 
\sH$ (otherwise, every even number $n \geq 4$ is in $\sH$, and we would only 
have $2$  single leaps, a contradiction). So we have the remaining single leap 
$(5, 6)$, and all leaps from his point on must be double. On the 
other hand, we should have $8=4+4 \in \sH$. Thus,
$\N \setminus \sH = \{1, 2, 3, 5, 6\}$ (3).

Finally, if $4 \not\in \sH$, then all $3$ single leaps occur on gaps $1, 2, 3, 
4$ and  all other leaps are double. Thus, $5 \in \sH$ and 
the only possibilities are 
$\N \setminus \sH = \{1, 2, 3, 4\}$ (4), $\N \setminus \sH = \{1, 2, 3, 4, 6\}$ (5) and
$\N \setminus \sH = \{1, 2, 3, 4, 6, 8\}$ (6).
\end{proof}
 
Notice that the previous result implies, in particular, that, if $\sH$ is a sparse 
semigroup having genus $g \geq 4$ and $\l_{g} = 2g - 4$, then $g \le 7$. We will
see in the next section (Theorem \ref{genusQuotaSparseTheo}) that it is possible 
to generalize this fact (and the result about quasi-symmetric sparse semigroups)
in the sense that for any fixed $r$, there are 
only a finite number of (sorts of) sparse semigroups, which can be explicitly listed.

\begin{theorem}\label{charactOdd5}
Let $\sH$ be a sparse semigroup having genus $g \geq 5$ and $\l_{g} = 2g - 5$. Then $\sH$ is one of the following:
\begin{enumerate}
\item $\sH = 3\N +\sH_{11}$, and $\sH$ is $4$-hyperelliptic;
\item $\sH = 3\N + \sH_{13}$, and $\sH$ is $4$-hyperelliptic;
\item $\sH = 2(\N \setminus \{1, 3\}) \cup \sH_{2g -4}$, with
 $g \ge 6$, and $\sH$ is $2$-hyperelliptic;
\item $\sH = \{ 0, 5, 7\} \cup \sH_{9}$,  and $\sH$ is $4$-hyperelliptic;
\item $\sH = \{ 0, 5, 7, 10\} \cup \sH_{11}$, and $\sH$ is $4$-hyperelliptic;
\item $\sH = \{ 0, 5, 7, 10, 12\} \cup \sH_{13}$, and $\sH$ is $4$-hyperelliptic;
\item $\sH = \{ 0, 5\} \cup \sH_{7}$, and $\sH$ is $3$-hyperelliptic;
\item $\sH = \{ 0, 5, 8\} \cup \sH_{9}$, and $\sH$ is $3$-hyperelliptic;
\item $\sH = \{ 0, 5, 8, 10\} \cup \sH_{11}$, and $\sH$ is $4$-hyperelliptic;
\item $\sH = 2(\N \setminus \{1, 2\}) \cup \sH_{2g -4}$, with $g \ge 5$, and $\sH$ is $2$-hyperelliptic;
\end{enumerate}
\end{theorem}

\begin{proof}
The proof technique is very similar to the one in the previous theorem. 

Theorem \ref{leapthm} (3) tells us that there are $S=4$. Again, $\l_{1} = 1$ and $\l_{2} = 2$, which 
account for one single leap.
If $3 \in \sH$, since $\sH$ is sparse, we must have that 
$1, 2, 4, 5, 7, 8, 10, 11 \not\in \sH$, accounting for all $4$ single leaps. Thus, either 
$\N \setminus \sH = \{1, 2, 4, 5, 7, 8, 10, 11\}$ and $g = 8$ (1), or 
$\N \setminus \sH = \{1, 2, 4, 5, 7, 8, 10, 11, 13\}$ and $g = 9$ (2). 

Otherwise, $\l_{3} = 3$, which accounts for a second single leap. If $4 \in \sH$, since $\sH$ is sparse, 
we must have that, for every $n \in 4\N$, $n < \l_{g}$, both $n - 1$ and $n + 1$ are gaps. In particular, 
$5, 7 \not\in \sH$. Notice that $6 \not\in \sH$ (otherwise, every even number $n \geq 4$ is in $\sH$, and 
we would only have $2$ single leaps, a contradiction). So we have other $2$ single leaps 
(as $5, 6, 7 \not\in \sH$), and all leaps from this point on must be double. Thus, 
$\N \setminus \sH = \{1, 2, 3, 5, 6, 7, 9, 11, \ldots, 2g - 5\}$ (and $2g-5 \ge 7$, so $g \ge 6$) (3).

Otherwise, $\l_{4} = 4$, which accounts for a third single leap. If $5 \in \sH$, since $\sH$ is sparse 
and $5 < \l_{g}$, $6 \not\in \sH$ and exactly one of the following possibilities holds: $7 \not\in \sH$ 
or $8 \not\in \sH$ (the two cannot happen simultaneously, as we would then have $6$ single leaps). 

If $7 \in \sH$, then $8 \not\in \sH$ and, since $10=5+5 \in \sH$ and $9 \le \l_{g}$, $9 \not\in \sH$. This 
accounts for the fourth and last single leap of $\sH$. So all leaps from his point on must be double, 
which means that all even numbers greater than $9$ must be in $\sH$. Since $14, 15 \in \sH$ and $\sH$ 
is sparse, $\{n \geq 14\} \subset \sH$. Thus, in this case, we have the following possibilities for 
$\sH$: $\sH = \{0, 5, 7\} \cup \{n \geq 10\}$ (4), $\sH = \{ 0, 5, 7, 10\} \cup \{n \geq 12\}$ (2) and 
$\sH = \{ 0, 5, 7, 10, 12\} \cup \{n \geq 14\}$ (6). 

If $8 \in \sH$, then $7 \not\in \sH$, and all single leaps occur on gaps $1, 2, 3, 4, 6, 7$. Thus, all leaps 
from this point on must be double, which means that all even numbers greater than $7$ must be in 
$\sH$. So $13=5+8 \in \sH$, and thus, in this case, we have the following possibilities for 
$\sH$: $\sH = \{ 0, 5\} \cup \{n \geq 8\}$ (7), $\sH = \{ 0, 5, 8\} \cup \{n \geq 10\}$ (8) and 
$\sH = \{ 0, 5, 8, 10\} \cup \{n \geq 12\}$ (9).

On the other hand, if $5 \not\in \sH$, then all $4$ single leaps occur on gaps $1, 2, 3, 4, 5$ and all other 
leaps are double. Thus, $$\N \setminus \sH = \{1, 2, 3, 4, 5, 7, 9, \ldots, 2g - 5\}$$ 
(here $2g-5 \ge 5$, and so $g \ge 5$) (10).
\end{proof}

An important feature of Theorems \ref{charactOdd3} and \ref{charactOdd5} will be
generalized in the next section (Corollary \ref{genusQuotaSparseOddTheo}): if 
$\sH$ is a sparse semigroup having genus $g$ for which $\l_g=2g-(2r+1)$, where 
$r \in \N$, then, if $g>4r+1$, all nongaps of $\sH$ smaller than $\l_g$ are 
necessarily even. In this case all the $(\l_g+1)/2=g-r$ odd positive integers 
smaller than $\l_g+1$ are gaps, so there are $r$ even gaps (i.e., $\sH$ is 
$r$-hyperelliptic), and the set $\{m \in \N \, ; \, 2m \in \sH\}$ is a semigroup of 
genus $r$.

Theoretically, one could use the proof technique in Theorems \ref{charactOdd3}
and \ref{charactOdd5} (and/or the above remark) to characterize all sparse
semigroups having $\l_{g} = 2g - (2r +1)$ for any fixed $r \in \N$. However, 
as $r$ grows, the number of cases to be analyzed quickly add up, and the
arguments in the proof become more intricate.

\section{Limit sparse semigroups}\label{limitsparse}

Having in mind the results of the previous section, it is only natural for one
to ask about the existence of sparse semigroups of genus
$g \geq 4r-1$ and Frobenius number $\l_g=2g-2r$ (where $\kappa=2r$). 
A preliminary analysis of examples suggests that such semigroups do not exist if
$g > 4r - 1$. 
This fact together with Lemma
\ref{reductionLemma} reinforces the idea that sparse semigroups of genus $g=4r-1$ with
even Frobenius number $\l_{g} = 2g - 2r$ are special. 

\begin{remark}
We note that Theorem \ref{leapthm} assures us that a sparse semigroup has genus $g=2\K-1$ if
 and only if $S = D$.
 \end{remark}

We call a sparse semigroup with as many single
as double leaps (equivalently, of genus $g=2\K-1$) a \emph{limit sparse semigroup}.

Starting from searching for an upper bound to the genus of limit sparse semigroups 
with even Frobenius number, our aim is to analyze more closely the structure of such semigroups 
regardless of its Frobenius number's parity. When needed, we shall
denote $\l_{g} = 2g - 2r$ ($\K=2r$) for even Frobenius numbers, or $\l_{g} = 2g - 2r -1$ ($\K=2r+1$)
for odd Frobenius numbers. 

\begin{lemma}\label{bijSDHLemma}
If $\sH$ is a limit sparse semigroup, then
$\# \{ \sH \cap \{1, 2, 3, \ldots, \l_{g}\} \} = S = D$.
\end{lemma}

\begin{proof}
Clearly, there are $\l_{g} = 2g - \K$ natural numbers in the set
$\{1, \ldots, \l_{g}\}$, $g$ of which are gaps. So, the $g - \K$ remaining ones are all in
$\sH \cap \{1, \ldots, \l_{g}\}$ and, thus,
$\# \{ \sH \cap \{1, 2, 3, \ldots, \l_{g}\} \} = g - \K = D =  S$.
\end{proof}

\begin{lemma}\label{4isGapLemma}
If $\sH$ is a limit sparse semigroup with even Frobenius number $\l_g=2g - 2r$.
Then $4 \not\in \sH$. 
\end{lemma}

\begin{proof}
We know that $\l_g=6r-2$, $\l_{g-1}=6r-4$, so $3r-2$, $3r-1$ are also gaps.
Assume, by contradiction, that $4 \in \sH$, then, analyzing the residual class
of $r$ modulo $4$, one of the integers $6r-2$, $3r-2$, $3r-2$ is a nongap, a contradiction.
\end{proof}

Let us now see that for each $r$ there is precisely one limit sparse semigroup with
even Frobenius number. First, we state and prove another technical lemma:

\begin{lemma}\label{3isNongapLemma}
Let $\sH$ be a limit sparse semigroup with even Frobenius number
$\l_{g}= 2g-2r = 6r-2$. Then $3 \in \sH$ if and only if $6r-5 \notin \sH$.
\end{lemma}

\begin{proof}
It is clear that if $3 \in \sH$ then $\sH=3\N+\sH_{6r-2}$, so
$6r-5\notin\sH$.

Now, let us assume that $6r-5 \notin \sH$. Since $S = D$ and because of
Lemma \ref{bijSDHLemma}, there are $2r-2$ nongaps in the interval $[1,6r-5]$.
Then for each such nongap $n \in \sH$, the consecutive numbers $6r-5-n$ and
$6r-4-n$ are gaps, which produces $2r-2$ single leaps. Notice that the single
leaps $(6r-5-n, 6r-4-n)$ are disjoint, for every $n \in \sH$, because of the
sparse property.

Suppose, by contradiction,  that $3 \notin \sH$. Thus $(1,2)$ and $(2,3)$ are
single leaps, and $(6r-5-n,6r-4-n)$ are $2r-2$ disjoint single leaps, for every
$n \in \sH \cap [1,6r-5]$, and so at least one of the two single leaps $(1,2)$
and $(2,3)$ are not of this form. Hence, since $(6r-5,6r-4)$ is also a single
leap, the number of single leaps is bigger than $2r-1$, which is a contradiction.
\end{proof}

\begin{theorem}\label{charactSparseLimTheo}
Let $\sH$ be a limit sparse semigroup  with even Frobenius number
$\l_g=2g-2r$. Then  $\sH= 3\N + \sH_{6r-2}$.
\end{theorem}

\begin{proof}
In view of Lemma \ref{3isNongapLemma}, all we have to do is to show that $6r-5$
is a gap. Suppose otherwise. So, there is an integer $x \geq 3$ such that the
last single leap is $(6r-2x-1,6r-2x)$. So, in the interval $[6r-2x,6r-2]$,
all even numbers are gaps and all odd numbers are nongaps.

Let $n_1$ be the multiplicity of $\sH$. We first notice that it cannot be that
$n_1$ is an odd number and $n_1 < 2x-1$. Otherwise, $6r-2-n_1$ would be an
odd gap and $6r-2-n_1 > 6r-2x-1$, a contradiction. Moreover, it cannot be
that $n_1$ is an even number and $n_1 \leq 2x$.
Otherwise, there would be an even nongap, a multiple of $n_1$, between
$6r-2x$ and $6r-2$, which is also a contradiction. Then there are only two
possibilities for $n_1$:
\begin{enumerate}
 \item[(A)] $n_1 > 2x$;
 \item[(B)] $n_1 = 2x-1$.
\end{enumerate}

Since the nongaps in the interval $[6r-2x,6r-2]$ are the odd numbers, there
are $2r-x$ positive nongaps smaller than $6r-2x-1$. For each one of them,
say $n \in \sH \cap [1,6r-2x-2]$, $6r-2x-1-n$ and $6r-2x-n$ are consecutive
gaps, producing $2r-x$ single leaps.

Let us consider each case separately: 
 
Consider (A). In this case, $1,2,3,\ldots,n_1-1$ are consecutive gaps. The single leaps
$(6r-2x-1-n,6r-2x-n)$ are all disjoint for every $n \in \sH$, $n < 6r-2x-1$.
Then, among the single leaps in the interval $[1,n_1-1]$, there are at least
$x-1$ different from the ones in
$\{ (6r-2x-1-n,6r-2x-n) \, ; \, n \in \sH, \, n < 6r -2x-1 \}$. Since
$(6r-2x-1,6r-2x)$ is also a single leap, we have, so far, $(x-1)+(2r-x)+1=2r$
single leaps, which is a contradiction, because $S=2r-1$.

Now, let us suppose (B). Since $S = 2r -1$, there are exactly $x-2$ single leaps in the interval
$[1,n_1-1]$ that are not  in 
$\{ (6r-2x-1-n,6r-2x-n) \, ; \, n \in \sH, \, n < 6r -2x-1 \}$. Thus,
there are $x-1$ single leaps in the interval $[1,2x-2]$ of the type 
$(6r-2x-1-n,6r-2x-n)$, where $n\in\sH\cap[1,6r-2x-2]$, which are 
necessarily $(1,2), (3,4), \ldots, (2x-3,2x-2)$. Hence  
$6r-4x+2, \ldots, 6r-2x-4, 6r-2x-2 \in \sH$. In particular, since
$x \ge 3$, $6r-2x-4 \in \sH$. 

Now, consider the interval $[2x-1,6r-4x+1]$ which has $3(2r-2x+1)\ne 0$
 integers and contains $2r-x-(x-1)=2r-2x+1$ nongaps. For each nongap 
$n \in \sH \cap [2x-1,6r-4x+1]$, we consider the consecutive gaps 
$6r-2x-1-n$ and $6r-2x-n$, which both belong to $[2x-1,6r-4x+1]$ (and 
thus are $2r-2x-1$ single leaps contained in this interval). Then, since
 $\sH$ is sparse, the nongaps in the interval $[2x-1,6r-4x+1]$ are 
necessarily the integers congruent to $2x-1 \pmod 3$. Hence, $2x+2\in \sH$
(this is clear when the interval $[2x-1,6r-4x+1]$ has more than $3$ elements
-- it has at least $3$ elements, since $3(2r-2x+1)$ is a positive multiple
of $3$; if $3(2r-2x+1)=3$, then $2x+2=6r-4x+2 \in \sH$). Thus, in this case,
$(2x+2)+(6r-2x-4)=6r-2=\l_g$ would be a nongap, a contradiction.

So, in any case, we arrive at a contradiction, and conclude that $6r-5$ is a
gap, which implies, by Lemma \ref{3isNongapLemma}, that $3 \in \sH$ and it is
straightforward to see that $\sH= 3\N + \sH_{6r-2}$.
\end{proof}

\begin{corollary}
If $\sH$ is a limit sparse semigroup with even Frobenius
number $\l_g=2g-2r$, then
$\sH$ is an Arf semigroup. 
\end{corollary}

We now have a tight bound for the genus of a sparse semigroup with even Frobenius
number, greatly improving that by Munuera, Torres and Villanueva
($g\leq 6r-n_1$ if $g \geq 4r-1$ \cite[Theorem 3.1]{MunueraTorresVillanueva}):

\begin{corollary}\label{genusQuotaSparseTheo}
Let $\sH$ be a sparse semigroup of genus $g$ with even Frobenius number $\l_g=2g-2r$.
Then $g \leq 4r - 1$.
\end{corollary}
\begin{proof}
Suppose $\sH$ is a sparse semigroup of genus $g = 4r+ j$, $j \geq 0$, with 
$\l_g=2g-2r$. Then, by Lemma \ref{reductionLemma}, there exists a sparse 
semigroup $\widetilde{\sH}$ of genus $\widetilde{g} = 4r - 1$ and 
$\l_{\widetilde{g}} = 6r - 2$ such $\sH$ is a subsemigroup of $\widetilde{\sH}$. 
Theorem \ref{charactSparseLimTheo} tells us that $\widetilde{\sH} =3\N+\sH_{6r-2}$.
 By the construction of $\widetilde{\sH}$ in the proof of 
Lemma \ref{reductionLemma}, we see that $\{n \in \N \, ; \, n < 6r \} \cap \sH = 
\{n \in \N \, ; \,  n < 6r \} \cap \widetilde{\sH}$; in particular, $3 \in \sH$ 
and thus $6r \in \sH$, which contradicts the fact that $6r = \l_{4r} = \l_{g - 
j}$ is a gap of $\sH$.
\end{proof}

Let us now analyze some constraints on limit sparse semigroups having
odd Frobenius number. 

\begin{theorem}\label{oddFrobEvenMultTheo}
Let $\sH$ be a limit sparse semigroup with odd Frobenius number
$\l_{g}= 2g - (2r+1) = 6r+1$. If the multiplicity $n_1$ of $\sH$ is even, then
every nongap smaller then $\l_g$ is even, and so $\sH$ is $r$-hyperelliptic.
\end{theorem}

\begin{proof}
Assume there is an odd nongap $x \in \sH$, $x < \l_g$. Let $\tilde n$ be the
largest odd nongap smaller than $\l_g$. The interval $(\l_g-n_1,\l_{g}]$ 
contains a complete residue system of the integers module $n_1$, so
$\tilde n$ satisfies $\l_g-n_1<\tilde n<\l_g$. By construction, $\tilde n+1$ and 
$\tilde n+2$ are consecutive gaps. So, since $\sH$ is sparse, there are less 
than $n_1/2$ nongaps in the interval $(\l_{g}-n_1,\l_{g}]$, thus there are at 
least $2r - n_1/2 +1$ positive nongaps smaller than $\l_{g}-n_1$.
Define $\Gamma:=\{n\in\sH\ ; 0 < n < \l_g-n_1\}$.  Being $\tilde n+1$ and 
$\tilde n+2$ consecutive gaps, for each $n\in\Gamma$, $\tilde n+1-n$ and 
$\tilde n+2-n$ are also consecutive gaps. But in the interval $[1,n_1-1]$ 
there are at least $n_1/2-1$ single leaps distinct from each 
$(\tilde n+1-n,\tilde n+2-n)$. By considering also the single leap 
$(\tilde n+1,\tilde n+2)$ we already count a number of 
$(2r - n_1/2 + 1) + (n_1/2-1) + 1 = 2r + 1$ single leaps, a contradiction. Now, the odd 
gaps are all the odd numbers between $1$ and $\l_g=6r+1$, and so there are 
$3r + 1$ odd gaps and $r$ even gaps, then $\sH$ is $r-$hyperelliptic.
\end{proof}

\begin{theorem}\label{oddFrobOddMul}
Let $\sH$ be a limit sparse semigroup with odd Frobenius number $\l_{g}= 2g - (2r+1) = 6r+1$.
If the multiplicity $n_1$ of $\sH$ is odd, then $\sH$ is one of the following:
\begin{enumerate}
\item $\sH = 3\N +\sH_{6r+1}$; 
\item $\sH = \langle 2j+1; j \in \N, r \le j \le 2r-1\rangle \cup \sH_{6r+1}$, with $r>1$.
\end{enumerate}
\end{theorem}

\begin{proof}
Let us first show that, if $6r-2 \notin \sH$, then $3 \in \sH$ 
(and so $\sH = 3\N +\sH_{6r+1}$). If $6r-2 \notin \sH$, 
then, since $S = D$ and because of Lemma \ref{bijSDHLemma}, there are $2r-1$ 
nongaps in the interval $[1,6r-2]$. Then for each such nongap $n \in \sH$, the 
consecutive numbers $6r-2-n$ and $6r-1-n$ are gaps, which produces $2r-1$ single 
leaps. Notice that the single leaps $(6r-2-n, 6r-1-n)$ are all disjoint, for every
$n \in \sH$, because of the sparse property.

Suppose, by contradiction,  that $3 \notin \sH$. Thus $(1,2)$ and $(2,3)$ are single
leaps, and $(6r-2-n,6r-1-n)$ are $2r-1$ disjoint single leaps, for every 
$n \in \sH \cap [1,6r-2]$, and so at least one of the two single leaps $(1,2)$ and 
$(2,3)$ are not of this form. Hence, since $(6r-2,6r-1)$ is also a single leap, the 
number of single leaps is bigger than $2r$, which is a contradiction.

Suppose now that $6r-2$ is a nongap. So, there is an integer $x \geq 2$ such that 
the last single leap is $(6r-2x,6r-2x+1)$. So, in the interval $[6r-2x+1,6r+1]$, 
all odd numbers are gaps and all even numbers are nongaps.

First notice that it cannot be that $n_1 < 2x+1$. Otherwise, $6r+1-n_1$ would be an 
even gap and $6r+1-n_1 > 6r-2x$, a contradiction. Then there are only two possibilities 
for $n_1$:
\begin{enumerate}
 \item[(A)] $n_1 > 2x+1$;
 \item[(B)] $n_1 = 2x+1$.
\end{enumerate}

Since the nongaps in the interval $[6r-2x+1,6r+1]$ are the even numbers, there are $2r-x$
positive nongaps smaller than $6r-2x$. For each one of them, say 
$n \in \sH \cap [1,6r-2x-1]$, $6r-2x-n$ and $6r-2x+1-n$ are consecutive gaps, producing 
$2r-x$ single leaps.

Let us consider each case separately: 

If we assume (A), then $1,2,3,\ldots,n_1-1$ are consecutive gaps. The single leaps 
$(6r-2x-n,6r-2x+1-n)$ are all disjoint for every $n \in \sH$, $n < 6r-2x$. Then, among the
single leaps in the interval $[1,n_1-1]$, there are at least $x$ different from the ones in
$\{ (6r-2x-n,6r-2x+1-n) \, ; \, n \in \sH, \, n < 6r - 2x \}$. Since $(6r-2x,6r-2x+1)$ is 
also a single leap, we have, so far, $x+(2r-x)+1=2r+1$ single leaps, which is a contradiction, 
because $S=2r$. 

Let us assume (B).
Since $S = 2r$, there are exactly $x-1$ single leaps in the interval $[1,n_1-1]$ that are
not  in $\{ (6r-2x-n,6r-2x+1-n) \, ; \, n \in \sH, \, n < 6r -2x \}$. Thus, there are $x$
single leaps in the interval $[1,2x]$ of the type $(6r-2x-n,6r-2x+1-n)$, where 
$n\in\sH\cap[1,6r-2x-1]$, which are necessarily $(1,2), (3,4), \ldots, (2x-1,2x)$. Hence
$6r-4x+1, \ldots, 6r-2x-3, 6r-2x-1 \in \sH$. In particular, since $x \ge 2$, $6r-2x-3 \in \sH$,
and, by the sparse property, $6r-4x$ is a gap. 

Now, consider the interval $[2x+1,6r-4x]$ which has $6(r-x)$ integers and contains 
$2r-x-x=2r-2x$ nongaps. For each nongap $n \in \sH \cap [2x+1,6r-4x]$, we consider the 
consecutive gaps $6r-2x-n$ and $6r-2x+1-n$, which both belong to $[2x-1,6r-4x+1]$ (and
thus are $2r-2x$ single leaps contained in this interval). Then, since $\sH$ is sparse,
the nongaps in the interval $[2x+1,6r-4x]$ are necessarily the integers congruous to 
$2x+1 \pmod 3$. Hence, provided $r-x \neq 0$ (in which case the interval $[2x+1,6r-4x]$
has $6(r-x) \ge 6$ integer elements), we have $2x+4\in \sH$. Thus, in this case, 
$(2x+4)+(6r-2x-3)=6r+1=\l_g$ would be a nongap, a contradiction. Thus we must have $r-x=0$,
i.e., $x=r$, and so $6r-4x+1=2x+1$ and $6r-2x+2=4x+2$. In this case, 
$\sH = \{0\}\cup \{2j+1; j \in \N, r \le j \le 2r-1\} \cup \{2j; j \in \N, 2r+1 \le j \le 3r\}\cup \{n \in \N \, ; \, n \geq 6r+2\}$,
which is clearly the semigroup $\sH = \langle 2j+1; j \in \N, r \le j \le 2r-1\rangle \cup \sH_{6r+1}$. 
\end{proof}

\begin{remark}\label{nonArf}
The semigroups of the form $\sH = 3\N + \sH_{6r+1}$ are clearly Arf,
but the semigroups of the form 
$\sH = \langle 2j+1; j \in \N, r \le j \le 2r-1\rangle + \sH_{6r+1}$
with $r>1$ are not Arf, since $4r-3, 4r-1 \in \sH$, but $2.(4r-1)-(4r-3)=4r+1 \notin \sH$.
\end{remark}

\begin{corollary}\label{genusQuotaSparseOddTheo}
Let $\sH$ be a sparse semigroup of genus $g$ with $\l_g=2g-(2r+1)$. Then either $g \leq 4r + 1$, or
all nongaps smaller than $\l_g$ are even.
\end{corollary}

\begin{proof}
Suppose $\sH$ is a sparse semigroup of genus $g = 4r + j=2(2r+1)+(j-2)$, $j \geq 2$, with
$\l_g=2g-(2r+1)$. Then, by Lemma \ref{reductionLemma}, there exists a sparse semigroup 
$\widetilde{\sH}$ of genus $\widetilde{g} = 4r + 1$ and $\l_{\widetilde{g}} = 6r + 1$ such
$\sH$ is a subsemigroup of $\widetilde{\sH}$. The previous theorems tell us that we have the
following possibilities:
\begin{enumerate}

\item[(A)] All nongaps of $\widetilde{\sH}$ (and so of $\sH$) smaller than 
$\l_{\widetilde{g}} = 6r + 1$ are even 

In this case, by the construction of $\widetilde{\sH}$ in the proof of Lemma \ref{reductionLemma},
the gaps of $\sH$ larger than $6r$ are given by a sequence of double leaps starting at $6r+1$,
and so all the remaining nongaps of $\sH$ smaller than $\l_g$ are even.
 
\item[(B)] $\widetilde{\sH} =3\N+\sH_{6r+1}$
  
Here, by the construction of $\widetilde{\sH}$ in the proof of Lemma \ref{reductionLemma},
we see that 
$\{n \in \N \, ; \, n < 6r+2 \} \cap \sH = \{n \in \N \, ; \,  n < 6r+2 \} \cap \widetilde{\sH}$; 
in particular, we have $6r+3 \in \sH$, which contradicts the fact that 
$6r+3 = \l_{4r+2} = \l_{g - j+2}$ is a gap of $\sH$.

   \item [(C)] $\widetilde{\sH} = \langle 2j+1; r \le j \le 2r-1\rangle \cup \sH_{6r+1}$. 

Since $2r+1 \in \sH$ we have $6r+3=3(2r+1) \in \sH$ -- which also contradicts the fact that $6r+3 = \l_{4r+2} = \l_{g - j+2}$ is a gap of $\sH$.
   
\end{enumerate}
   
\end{proof}

\section{On sparse Weierstrass semigroups}

One of the main applications of numerical semigroups is the study of Weierstrass points on curves.
Having dealt (and classified limit) sparse semigroups, it is natural to ask whether those are realized as a Weierstrass semigroup. 

By a \emph{curve}, we mean a smooth projective curve defined over an algebraically
closed field of characteristic zero. If $\cX$ is a curve and $P\in\cX$ is a
point of $\cX$, then the Weierstrass semigroup $H(P)$ of the pair $(\cX,P)$
consists of the integers $n$ for which there does exists a meromorphic
function on $\cX$ with pole divisor $nP$, i.e.
$$H(P):=\{n\in\N \, ; \,dim H^0(\cX\, ,\SO_{\cX}((n-1)P))< \dim H^0(\cX\, ,\SO_{\cX}(nP))\}$$
It is clear that $H(P)$ is a numerical semigroup. Now, it follows from Riemann-Roch Theorem
that the gap sequence of $H(P)$ is the sequence of orders of vanishing 
of the holomorphic differentials of $\cX$ at $P$. In particular, the genus of $H(P)$
is equal to the genus of the curve $\cX$. A numerical semigroup is a \emph{Weierstrass semigroup}
if it is realized as Weierstrass semigroup of some pair $(\cX,P)$. A point on a curve
is a \emph{Weierstrass point} if its associated Weierstrass semigroup is different from $\sH_{g}$, where
$g$ is the genus of the curve.

It is known, from Rim--Vitulli \cite{Rim-Vitulli}, that the negatively graded semigroups are only
the ordinary, hyperordinary and those semigroups of multiplicity $m>1$ having precisely one gap
between $m$ and $2m$.
By a theorem of  Pinkham \cite{Pinkham}, we know that a monomial curve associated to a negatively graded semigroup can be smoothed.
In particular, a (hyper)ordinary semigroup is a Weierstrass semigroup. Thus, from those works and Theorem \ref{charactSparseLimTheo}, we get:

\begin{corollary}
Let $\sH$ be a limit sparse semigroup of genus $g$.
If the Frobenius number $\l_g$ is even, then $\sH$ is a Weierstrass semigroup;
\end{corollary}

It is clear that if $\sH$ is a limit sparse semigroup
with odd Frobenius number of the type $\sH = 3\N+\sH_{6r+1}$ (see Theorem \ref{oddFrobOddMul}),
then $\sH$ is also a Weierstrass semigroup. 

\begin{corollary}
Let $\sH'$ be any numerical semigroup of genus $r$. Consider 
$\sH:=2\sH' \cup \sH_{6r+1}$. Then $\sH$
is a limit sparse semigroup with odd Frobenius number with even
multiplicity (see Theorem \ref{oddFrobEvenMultTheo}).
Reciprocally, every semigroup satisfying the hypothesis of Theorem \ref{oddFrobEvenMultTheo}
 arises in this way. Moreover, $\sH$ is Arf if and only if $\sH'$ is Arf. 

\end{corollary}

\begin{proof}
It is clear that if $\sH'$ is a semigroup of genus $r$, then
the Frobenius number of $\sH:=2\sH' \cup \sH_{6r+1}$ 
is $6r+1$. The gaps of $\sH$ are all the odd integers in $[1,6r+1]$
and the even numbers $2m$ where $m$ is a gap of $\sH'$.
Thus the genus of $\sH$ is $4r+1$. The sparse condition follows
from the construction of $\sH$.
On the other hand, let $\sH$ be a sparse semigroup of genus
$g=4r+1$, with Frobenius number $\l_g=2g-(2r+1)$, and even
multiplicity $n_1$. Thus Theorem \ref{oddFrobEvenMultTheo}
ensures that all nongaps smaller than $\l_g=6r-1$ are even.
The odd gaps are all the odd numbers between $1$ and $\l_g=6r+1$. 
So there are $3r+1$ odd gaps and $r$ even gaps. Thus
$\sH\cap [1,\dots,\l_{g}]=\{2\,n_1,2\,n_2,\dots,2\,n_{2r}\}$.
We consider the set $\sH':=\{0,n_1,\dots,n_{2r},n_{2r}+1,n_{2r}+2,\dots\}$.
Since $\sH$ is additively closed, $\sH'$ a semigroup.
An integer is a gap of $\sH'$ if and only if $2n$ is a gap of
$\sH$. Thus the genus of $\sH'$ is $r$. The Arf condition follows
from the construction of $\sH'$.
We notice that the last gap in $\sH'$ is at most $2r-1$, and so the last
even gap of $\sH$ is at most $4r-2$.
\end{proof}

Following the same steps of the above Corollary, we get:

\begin{corollary}\label{rhiperg4r1}
Let $\sH$ be a semigroup as in Corollary \ref{genusQuotaSparseOddTheo}, with
$g>4r+1$. Then $\sH$ is obtained in the following way: take any numerical 
semigroup $\sH' \subset \N$ of genus $r$, and
set $\sH:=2\sH' \cup \sH_{2g-2r-1}$. 

\end{corollary}

\begin{proof}
If $\sH$ is a numerical semigroup as in Corollary \ref{genusQuotaSparseOddTheo}, with $g>4r+1$, then 
its odd gaps are all the odd numbers between $1$ and $\l_g=2g-(2r+1)$. So, there are $g-r$ odd gaps and $r$ even gaps. Thus 
$\sH:=2\sH' \cup \sH_{2g-2r-1}$ where $\sH'$ is a numerical
semigroup of genus $r$. Note that the last gap of $\sH'$ is at most $2r-1$, and so the last even gap of
 $\sH$ is at most $4r-2$. \end{proof}

As can be noted of from the last two corollaries above, and the Corollary \ref{genusQuotaSparseOddTheo},
we may expect that some sparse semigroups with odd Frobenius number arise as a double 
covering of a genus $r$ curve. We recall that Torres \cite{Torres} characterized $r$-hyperelliptic curves 
of genus $g$ which arise as a double covering of a genus $r$-curves under the assumption $g\ge 6r+4$.
Gathering the Corollary \ref{rhiperg4r1} and the comment after the proof of Theorem A of \cite{Torres} we get:

\begin{corollary}
Let $\sH$ a numerical sparse semigroup of genus $g\ge 6r+4$ and Frobenius number $\l_{g}=2g-(2r+1)$.
If $\sH$ is a Weierstrass semigroup, then it arises as a double covering of a genus $r$-curve. In this case, we have $\sH:=2\sH' \cup \sH_{2g-2r-1}$, where $\sH'=\{n/2\mid\, n\in \sH \text{ is even}\}$ is a Weierstrass semigroup of genus $r$.
\end{corollary}

\begin{question}
Let $\sH' \subset \N$ any Weierstrass semigroup of genus $r$, and $g\ge 4r+1$.
Is the semigroup $\sH:=2\sH' \cup \sH_{2g-2r-1}$ always Weierstrass? 
\end{question}

\begin{question}
Let $\sH$ be a sparse Weierstrass semigroup with odd Frobenius number
$\l_{g}= 2g - (2r+1)$ and $g\ge 4r+1$ such that the multiplicity $n_1$ of $\sH$ is even. Is the semigroup $\sH'=\{n/2\mid\, n\in \sH \text{ is even}\}$ always Weierstrass? 
\end{question}

\begin{remark} 
If the answer to Problem A of \cite{Komeda} is affirmative then the answer to this last question is also affirmative (indeed it would be enough that there were
no numerical semigroup belonging to the box numbered by viii) in \cite{Komeda}).
\end{remark}

\begin{question}
Is the limit sparse semigroup $\sH = \langle 2j+1; j \in \N, r \le j
\le 2r-1\rangle \cup \sH_{6r+1}$ Weierstrass for every $r>1$?
\end{question}

\section*{Acknowledgements}
The authors warmly thank Prof. Luciane Quoos (Universidade Federal 
do Rio de Janeiro) for introducing sparse semigroups to them.


\section*{References}

 \begin{thebibliography}{99}
 
 \bibitem[Arf]{Arf}
C. Arf, \emph{Une interpretation alg\'ebrique de la suite  des ordres de multiplicit\'e d'une branche alg\'ebrique},
Proc. London Math. Soc., \textbf{50} (1949), 256 -- 287.

\bibitem[BC]{BallicoCentina}
E. Ballico, A. Del Centina, \emph{Ramification Points of double covering of curves and Weierstrass Points},
Annali di Matematica Pura ed Applica, \textbf{178}, (1999), 293--313.

\bibitem[BDF]{BarucciDobbsFontana}
V. Barucci, D. E. Dobbs, M. Fontana, \emph{Maximality properties in numerical semigroups and applications to one-dimensional analytically irreducible local domains},
Mem. Amer. Math. Soc. \textbf{125}, 1997.

\bibitem[Lip]{Lipman}
J. Lipman, \emph{Stable ideal and Arf semigroups}, Amer. J. Math. \textbf{97} (1975), 791 -- 813.

\bibitem[Kom]{Komeda}
J. Komeda, \emph{A generalization of Weierstrass semigroups on a double
covering of a curve}, Languages, Computations, and Algorithms in Algebraic Systems - Kyoto University Research Information Repository \textbf{1655} (2009), 124 -- 131. Available at
\hfil\goodbreak
\url{http://www.kurims.kyoto-u.ac.jp/~kyodo/kokyuroku/contents/pdf/1655-15.pdf}





\bibitem[MTV]{MunueraTorresVillanueva}
C. Munuera, F. Torres, J. Villanueva, \emph{Sparse Numerical Semigroups}, Lecture Notes in Computer Science: Applied Algebra, Algebraic Algorithms and Error-Correcting Codes,
\textbf{5527}, 23 -- 31, Springer-Verlag Berlin Heilderberg (2009). 

\bibitem[O]{Oliveira}
G. Oliveira, \emph{Weierstrass semigroups and the canonical ideal of non-trigonal curves},
Manuscripta Math., \textbf{71} (1991), 431--450.

\bibitem[P]{Pinkham} 
H.\ Pinkham, \emph{Deformations of algebraic varieties with $G\sb{m}$-action},
Ast\'erisque \textbf{20 }(1974), 1--131.

\bibitem[RV]{Rim-Vitulli}
D.S.\ Rim and M.A.\ Vitulli,
\emph{Weierstrass points and monomial curves},
J.\ Algebra \textbf{48 }(1977) 454--476.

\bibitem[RGB]{RosalesGarciasBranco}
J.C. Rosales, P. A. Garc\'ia-S\'anchez, J. I. Garc\'ia-Garc\'ia, M. B. Branco, \emph{Arf Numerical Semigroups},
Journal of Algebra \textbf{276} (2004), 3 -- 12.

\bibitem[T]{Torres}
F. Torres, \emph{Weierstrass points and double coverings of curves with applications: Symmetric numerical semigroups which cannot be realized as Weierstrass semigroups},
Manuscripta Math \textbf{83} (1994), 39 -- 58.

\bibitem[T2]{Torres2}
F. Torres, \emph{On $\gamma$-Hyperelliptic Numerical Semigroups}, Semigroup Forum  \textbf{55} (1997), 364 -- 379.

 \end{thebibliography}
\end{document}